\documentclass[]{amsart}
\usepackage[utf8]{inputenc}
\usepackage{graphicx}
\usepackage[usenames, dvipsnames]{color}

\DeclareUnicodeCharacter{2212}{-}

\usepackage{amsmath, amssymb}
\usepackage{amsthm}
\usepackage{enumerate}
\usepackage{mathtools}
\usepackage{tikz-cd}
\usepackage{comment}
\usepackage{color}
\usepackage{hyperref}
\usepackage{footmisc}

\newcommand{\kk}{{\bf k}}
\newcommand{\ak}{{\overline{\bf{k}}}}

\newcommand{\p}{\operatorname{\mathbb{P}}}
\newcommand{\C}{\operatorname{\mathbb{C}}}
\newcommand{\Z}{\operatorname{\mathbb{Z}}}

\newcommand{\GL}{\operatorname{GL}}

\newcommand{\aff}{\operatorname{Aff}}

\newcommand{\Bir}{\operatorname{Bir}}
\newcommand{\Cr}{\operatorname{Cr}}
\newcommand{\PGL}{\operatorname{PGL}}
\newcommand{\SL}{\operatorname{SL}}
\newcommand{\T}{T}

\newcommand{\cent}{\operatorname{Cent}} %Centraliser
\newcommand{\zenter}{\operatorname{C}} %Center
\newcommand{\norm}{\operatorname{Norm}}

\newcommand{\aut}{\operatorname{Aut}}

\newcommand{\PSL}{\operatorname{PSL}}

\newcommand{\BirDiff}{\operatorname{BirDiff}}

\newcommand{\Gal}{\operatorname{Gal}}

\newcommand{\s}{\mathcal{S}}

\newcommand{\A}{\operatorname{\mathbb{A}}}
\newcommand{\Q}{\operatorname{\mathbb{Q}}}
\newcommand{\R}{{\operatorname{\mathbb{R}}}}
\newcommand{\F}{\operatorname{\mathbb{F}}}
\newcommand{\CC}{\operatorname{\mathbb{C}}}

\newcommand{\id}{\operatorname{id}}
\newcommand{\car}{\operatorname{char}}

\def\dashmapsto{\mapstochar\dashrightarrow}

\def\dashmapsto{\mapstochar\dashrightarrow}

\renewcommand{\to}{\longrightarrow}
\newcommand{\rat}{\dashrightarrow}

\theoremstyle{plain}
\newtheorem{theorem}{Theorem}[section]
\newtheorem{lemma}[theorem]{Lemma}
\newtheorem{proposition}[theorem]{Proposition}
\newtheorem{corollary}[theorem]{Corollary}

\theoremstyle{definition}
\newtheorem{definition}[theorem]{Definition}
\newtheorem{example}[theorem]{Example}
\newtheorem{remark}[theorem]{Remark}

\address{EPFL SB MATH, 
	Station 8,
	CH-1015 Lausanne, Switzerland
}
\email{christian.urech@gmail.com}
\address{Laboratoire angevin de recherche en math\'ematiques (LAREMA), CNRS, Universit\'e d'Angers, 49045 Angers cedex 1, France}
\email{susanna.zimmermann@univ-angers.fr}


\thanks{
During this project, the first author was supported by the Swiss National Science Foundation project P2BSP2\_175008, the second author was supported by the ANR Project FIBALGA ANR-18-CE40-0003-01. 
Both authors received funding from Projet PEPS 2019 ``JC/JC" and were supported through the programme ``Research in Pairs" by the Mathematisches Forschungs Institut Oberwolfach 2019.
}

\begin{document}
	
	\author{Christian Urech and Susanna Zimmermann}
	\title{Continuous automorphisms of Cremona groups}.
	\maketitle

\begin{abstract}
	We show that if a group automorphism of a Cremona group of arbitrary rank is also a homeomorphism with respect to either the Zariski or the Euclidean topology, then it is inner up to a field automorphism of the base-field. Moreover, we show that a similar result holds if we consider groups of polynomial automorphisms of affine spaces instead of Cremona groups.
\end{abstract}

\section{Introduction}
The {\it Cremona group in $d$-variables} over a field $\kk$ is the group of birational transformations of the projective $d$-space  $\p^d_{\kk}$ over $\kk$. Equivalently, it can be seen as the group of $\kk$-automorphisms of the field $\kk(x_1,\dots, x_d)$. Cremona groups have been the object of research for over 150 years and considerable progress in the understanding of their structure has been achieved recently. This article concerns  group automorphisms of $\Cr_d(\kk)$.

An important subgroup of $\Cr_d(\kk)$ is the automorphism group $\aut(\p_{\kk}^d)$, which is isomorphic to $\PGL_{d+1}(\kk)$.  Every field automorphism $\alpha$ of $\kk$ naturally induces a group automorphism on both, $\Cr_d(\kk)$ and $\PGL_d(\kk)$, which we denote by $g\mapsto {}^\alpha\!g$. The group automorphisms of $\PGL_d(\kk)$ are well-known: every automorphism of $\PGL_{d+1}(\kk)$ is the composition of an inner automorphism with an automorphism of the form $g\mapsto {}^\alpha\!g$ or $g\mapsto  {}^\alpha\!g^\vee$, where $\alpha$ is a field automorphism of $\kk$ and $g^\vee$ denotes the inverse of the transpose of $g$. Note that this result implies in particular, that all group automorphisms of $\PGL_{d+1}(\kk)$ are Zariski continuous. 

In \cite{MR2278755}, D\'eserti showed  that every group automorphism  of  $\Cr_2(\C)$ is inner up to automorphisms induced by a field automorphism of $\C$. 
It is a natural question, whether the theorem of D\'eserti holds in all dimensions and for all fields. A difficulty in studying group automorphisms of $\Cr_d(\kk)$ for $d>0$ is that no easy to handle set of generators of the group is known. However, $\Cr_d(\kk)$ can be endowed with the {\it Zariski topology}, a topology that extends the Zariski topology of $\PGL_{d+1}(\kk)$ (see Section~\ref{sec:Zariski}). We generalize D\'eserti's theorem to arbitrary dimensions and arbitrary fields of chracteristic $0$ under the additional assumption that the group automorphisms are also homeomorphisms with respect to the Zariski topology. 

\begin{theorem}\label{main:cremona}
	Let $\kk$ be a field of characteristic $0$ and let $\varphi\colon\Cr_d(\kk)\to\Cr_d(\kk)$ be a  group automorphism  that is a homeomorphism with respect to the Zariski topology, where $d\geq 2$. Then there exists a field automorphism $\alpha$ of $\kk$ and an element $f\in\Cr_d(\kk)$ such that $\varphi(g)=f\, ({}^\alpha\! g) f^{-1}$ for all $g\in\Cr_d(\kk)$.
\end{theorem}

If $\kk=\R$ or $\kk=\C$, then $\Cr_d(\kk)$ can moreover be equipped with the Euclidean topology - a topology that makes $\Cr_d(\kk)$  a Hausdorff topological group and which restricts to the standard Euclidean topology on $\aut(\p^d_{\kk})$ (see Section~\ref{sec:Euclidean}). In this setting, Theorem~\ref{main:cremona} also holds if we consider the Euclidean topology instead of the Zariski topology:

\begin{theorem}\label{main:euclidean}
	Let $\kk=\C$ or $\kk=\R$ and let $\varphi\colon\Cr_d(\kk)\to\Cr_d(\kk)$ be a  group automorphism  that is a homeomorphism with respect to the Euclidean topology, where $d\geq 2$. Then there exists a field-automorphism $\alpha$ of $\kk$ that is Euclidean continuous and an element $f\in\Cr_d(\kk)$ such that $\varphi(g)=f\, ({}^\alpha\! g) f^{-1}$ for all $g\in\Cr_d(\kk)$.
\end{theorem}

The main ingredients for the proof of Theorem~\ref{main:cremona} and Theorem~\ref{main:euclidean} are a result of Cantat and Xie about embeddings of finite index subgroups of $\SL_{d+1}(\Z)$, the classification of automorphisms of $\PGL_{d+1}(\kk)$, and a topological deformation  argument.

D\'eserti's theorem implies that every group automorphism of $\Cr_2(\C)$ is Zariski continuous. It is therefore a natural question whether every group automorphism  of $\Cr_d(\kk)$ is Zariski continuous.

Observe that Theorem \ref{main:cremona} implies in particular, that if $\kk$ doesn't have any non-trivial field automorphisms (for instance, if $\kk=\Q$ or $\kk=\R$), then every continuous isomorphism of $\Cr_d(\kk)$ is inner.  In this setting, Theorem~\ref{main:euclidean} can be seen as an algebraic analogue of  \cite{filipkiewicz1982isomorphisms}, where it is shown that all group automorphisms of diffeomorphism groups are inner. Following this theme, we look in Section~\ref{sec:birdiff} at continuous automorphisms of the group of diffeomorphisms of $\p^d_{\R}(\R)$ that are induced by birational transformations.

In Section~\ref{sec:poly}, we consider instead of $\Cr_d(\kk)$ its sister, the group of polynomial automorphisms $\aut(\A^d_\kk)$, and obtain the following result:

\begin{theorem}\label{main:aut}
		Let $\kk$ be an infinite perfect field and let $\varphi\colon\aut(\A^d_\kk)\to\aut(\A^d_\kk)$ be a group automorphism that is a homeomorphism with respect to the Zariski topology. Then there exists a field automorphism $\alpha$ of $\kk$ and an element $f\in\aut(\A^d_\kk)$ such that $\varphi(g)=f\, ({}^\alpha\! g) f^{-1}$ for all $g\in\aut(\A^d_\kk)$.
\end{theorem}

The group $\aut(\A^d_\kk)$ has the additional structure of an {\it ind-group} (see \cite{furterkraft} for details). Theorem \ref{main:aut} implies in particular that every ind-group automorphism of $\aut(\A^d_\kk)$ is inner, if $\kk$ is an infinite perfect field. In the case where $\kk$ is of characteristic zero and algebraically closed, this fact was proven by Kanel-Belov, Yu, and Elishev \cite{kanel2018augmentation}. Again, one can ask the question, whether all group automorphisms of $\aut(\A^d_\kk)$ are Zariski continuous. In dimension 2, the question has a positive answer if $\kk$ is uncountable \cite{MR2209276}. A partial generalization of \cite{MR2209276} to higher dimensions has been obtained in \cite{kraftstampfli}, \cite{stampfli}, \cite{urech2013automorphisms}. 

It could be interesting to generalize D\'eserti's theorem as well as the results of this paper to more general fields, in particular to finite fields. Another natural question is in as far similar results hold if we drop either the condition of surjectivity or injectivity of the group endomorphisms; these questions have been raised in \cite{MR2292279}, \cite{MR3985695}, and also \cite{balnclamyzimmermann}. 

\vskip\baselineskip

{\bf Acknowledgements:} The authors would like to thank 
Jean-Philippe Furter, 
Hanspeter Kraft, and
Immanuel van Santen (n\'e Stampfli) for interesting discussions related to this project.

%%%%%
\section{Topologies on the Cremona groups}
 Cremona groups carry the so-called {\it Zariski topology}, which was introduced by M. Demazure \cite{MR0284446}. Over local fields, this topology can be refined to the {\it Euclidean topology} - a construction due to J. Blanc and J.-P. Furter \cite{BlancFurter}. Both topologies restrict to the usual Zariski or Euclidean topology on the group $\aut(\p^d_{\kk})$. In this section we will briefly recall the definitions and main properties of those topologies.

%%%%
\subsection{The Zariski topology}\label{sec:Zariski}

Let $\kk$ be a field and $X$ and $A$ irreducible algebraic varieties defined over $\kk$. Consider a birational map $f\colon A\times X\rat A\times X$ inducing an isomorphism between open dense subsets $U$ and $V$, where $U,V\subset A\times X$ have the property that the restriction of the first projection to $U$ and $V$ induces a surjective morphism onto  $A$. Every $\kk$-point $a\in A(\kk)$ induces a birational transformation $f_a$ of $X$ by mapping $x\in X$ to $p_2(f(a,x))$, where $p_2\colon A\times X\to X$ is the second projection. The map from $A(\kk)$ to $\Bir(X)$ given by $a\mapsto f_a$ is called a {\em morphism} (or {\em $\kk$-morphism}) from $A$ to $\Bir(X)$, and is denoted by $A\to\Bir(X)$. 

The Zariski topology is now the finest topology such that the preimages of closed subsets by morphisms are closed:

\begin{definition}
A subset $F\subset\Bir(X)$ is closed in the {\em Zariski topology} if for any algebraic $\kk$-variety $A$ and any $\kk$-morphism $A\to \Bir(X)$ the preimage of $F$ is closed in $A(\kk)$.
\end{definition}

Recall that, once we choose coordinates of $\p^d_{\kk}$, a Cremona transformation is given by $[x_0:\dots:x_d]\dashmapsto [f_0:\dots:f_d]$, where the $f_i\in\kk[x_0,\dots, x_d]$ are homogeneous polynomials of the same degree without a non-constant common factor. The {\it degree} of $f$ is then defined to be the degree of the $f_i$. We denote by $\Cr_d(\kk)_{\leq n}$ the set of Cremona transformations of degree $\leq n$ and  by $\Cr_d(\kk)_{n}$ the set of those of degree~$n$. 

J. Blanc and J.-P. Furter give an equivalent construction of the Zariski topology on the Cremona groups \cite[\S2]{BlancFurter} 
by defining a topology on $\Cr_d(\kk)_{\leq n}$ and extending it to $\Cr_d(\kk)$ as the inductive limit topology. The sets $\Cr_d(\kk)_n$ can be equipped canonically with the structure of a $\kk$-variety such that the elements in $\Cr_d(\kk)_n$ are exactly its $\kk$-points. However, the sets $\Cr_d(\kk)_{\leq n}$, while being closed, can not be given in any reasonable way the structure of an algebraic variety. This is part of the obstruction that $\Cr_d(\kk)$ can not be equipped with the structure of an ind-variety when endowed with the Zariski topology \cite[Theorem 2]{BlancFurter}.

Let $G$ be an algebraic group and $G\to\Cr_d(\kk)$ a morphism which is also a group homomorphism from the group $G(\kk)$ of $\kk$-points of $G$ to $\Cr_d(\kk)$. The image is called an {\it algebraic subgroup }of $\Cr_d(\kk)$. A subgroup of $\Cr_d(\kk)$ is an algebraic subgroup if and only if it Zariski closed and of bounded degree \cite[Corollary 2.18, Lemma 2.19]{BlancFurter}.

Note that $\kk$-morphisms of varieties into $\Cr_d(\kk)$ induce Zariski continuous maps from the $\kk$-points to $\Cr_d(\kk)$ by definition. Their image is of bounded degree, more generally, we have:

\begin{lemma}\label{lem:bounded_degree}
Let $A$ be an algebraic variety over  a field $\kk$ and $\varphi\colon A\to\Cr_d(\kk)$ a Zariski continuous map. Then $\varphi(A(\kk))$ is of bounded degree. 
\end{lemma}
\begin{proof} 
Assume, for a contradiction, that $\varphi(A(\kk))$ is of unbounded degree. Then there exists an infinite sequence of elements $\{f_n\}_{n\in\Z_{>0}}\subset\varphi(A(\kk))$ such that all the $f_i$ are of different degree. This implies that the induced topology on  $\{f_n\}_{n\in\Z_{>0}}$ is discrete and so in particular not noetherian. But this yields that the induced topology on $\varphi^{-1}(\{f_n\}_{n\in\Z_{>0}})$ is not noetherian, which is a contradiction, since a subspace of a noetherian space is noetherian.
\end{proof}

\subsection{The Euclidean topology}\label{sec:Euclidean}
We now introduce the Euclidean topology on $\Cr_n(\kk)$ as defined in \cite[\S 5]{BlancFurter}.

Let $\kk[x_0,\dots,x_d]_n$ be the $\kk$-vector space of homogeneous polynomials of degree $n$ and consider the projectivisation $W_n(\kk)=\p((\kk[x_0,\dots,x_d]_n)^{d+1})$, which is the set of $\kk$-rational points of a projective space $W_n$ of dimension $r=(d+1){d+n \choose d} -1$. 
An element $h=[h_0:\dots:h_d]\in W_n(\kk)$ induces the rational map of $\p_\kk^d$ given by $\psi_h\colon [x_0:\dots:x_d]\dashmapsto[h_0(x_0,\dots,x_d):\dots:h_d(x_0,\dots,x_d)]$, and 
the set $H_n(\kk)\subset W_n(\kk)$ of elements inducing a birational map of $\p_\kk^d$ is locally closed in $W_n(\kk)$. There exists a locally closed subvariety  $H_n\subset W_n$ such that $H_n(\kk)$ is exactly the set of $\kk$-points of $H_n$  \cite[Lemma~2.4(2)]{BlancFurter}.
The map $\pi_n\colon H_n(\kk)\to\Cr_d(\kk)_{\leq n}$, $h\mapsto \psi_h$ corresponds to a morphism $H_n\to\Cr_d(\kk)$. It is surjective, continuous and closed with respect to the Zariski topology, and is thus a topological quotient \cite[Corollary 2.9]{BlancFurter}.  It turns out that the Zariski topology on $\Cr_d(\kk)$ is the inductive limit topology given by the closed sets $\Cr_d(\kk)_{\leq d}$ endowed with the quotient topology \cite[Proposition~2.20]{BlancFurter}. 

\begin{definition}
If $\kk$ is a local field, we endow the $W_n(\kk)$ and $H_n(\kk)$ with the Euclidean topology and define the {\it Euclidean topology} on $\Cr_d(\kk)$ to be the inductive limit topology given by the $\Cr_d(\kk)_{\leq n}$ endowed with the quotient topology $\pi_n\colon H_n(\kk)\to\Cr_d(\kk)_{\leq n}$. 
\end{definition}

Endowed with the Euclidean topology, $\Cr_d(\kk)$ is a Hausdorff topological group which is not metrisable, and any compact subset is of bounded degree \cite[Theorem~3, Lemma 5.16, Lemma 5.13]{BlancFurter}. By definition, the Euclidean topology is finer than the Zariski topology and we have the following property:

\begin{lemma}[{\cite[Lemma 2.11]{Blanc:2015ab}}]\label{lem:morph_eucl_cont}
Let $A$ be an algebraic variety over a local field $\kk$ and $\rho\colon A\to \Cr_d(\kk)$ a $\kk$-morphism. Then the corresponding map $A(\kk)\to\Cr_d(\kk)$ is Euclidean continuous. 
\end{lemma}

%%%%
\section{A homomorphism of $\Cr_d(\kk)$ is determined by its restriction to $\aut(\p^n_\kk)$}

In this section, we show that for infinite fields a surjective Zariski or Euclidean continuous homomorphism of $\Cr_d(\kk)$ is determined by its restriction to $\aut(\p^d_\kk)$. The main ingredients for this result are continuous deformations of arbitrary Cremona transformations to linear maps. Consider the following example:

\begin{example}\label{ex:deformation}
Let $\kk$ be a field and pick $f\in\Cr_d(\kk)$ of degree $e$. 
With respect to affine coordinates on the chart $x_0\neq0$ we can write 
\[
f\colon(x_1,\dots,x_d)
\rat 
\left(F_1,\dots, F_d\right),
\]
where \[
F_i=\frac{P_{i0}+P_{i1}(x)+\cdots+P_{ie}(x)}{Q_{i0}+Q_{i1}(x)+\cdots+Q_{ie}(x)},
\]
for some homogeneous $P_{ij},Q_{ij}\in\kk[x_1,\dots,x_d]$ of degree $j$. 

For $t\in\kk\setminus\{0\}$, we define $\beta_t\in\aut(\p^d_\kk)$  by
\[
\beta_t\colon(x_1,\dots,x_n)\mapsto (tx_1,\dots,tx_d).
\]
The map $t\mapsto\beta_t^{-1} f\beta_t$ defines a $\kk$-morphism $\rho\colon\A_\kk^1\setminus\{0\}\to\Cr_d(\kk)$ whose image consists of conjugates of $f$ by linear maps. 
\end{example}

\begin{lemma}\label{lem:deformation_exists}
Let $\kk$ be an infinite field.
The $\kk$-morphism $\rho\colon\A_\kk^1\setminus\{0\}\to\Cr_d(\kk)$ from Example~\ref{ex:deformation} extends to a Zariski continuous map $\hat{\rho}\colon\A_\kk^1(\kk)\to\Cr_d(\kk)$ with $\hat{\rho}_0\in\aut(\p^n_\kk)$ if and only if $f$ is a local isomorphism at the point $p=[1:0:\cdots:0]$ and $f(p)=p$.  

The same statement holds if $\kk$ a local field and if we consider the Euclidean instead of the Zariski topology.
\end{lemma}

\begin{proof}
With respect to affine coordinates $x=(x_1,\dots, x_d)$ on the chart $x_0\neq0$ we can write 
\[
\rho_t(x)=\beta_t^{-1} f\beta_t(x_1,\dots,x_d)=\left(F^t_1,\dots, F^t_d\right),
\]	
where 
\[
F_i^t=\frac{t^{-1}P_{i0}+P_{i1}(x)+tP_{i2}(x)+\cdots t^{e-1}P_{ie}(x)}{Q_{i0}+tQ_{i1}(x)+\cdots t^eQ_{ie}(x)}.
\]
	
First, suppose that $f$ fixes $p=[1:0:\cdots:0]$ and is a local isomorphism at $p$. Then $P_{i0}=0$ and $Q_{i0}\neq0$ for $i=1,\dots,d$ and hence $\rho_0$ corresponds to the derivative (the linear part) of $f$ at $[1:0:\cdots:0]$. Thus $\rho$ extends to a $\kk$-morphism $\hat{\rho}\colon\A^1\to\Cr_d$ with $\hat{\rho}_0\in\aut(\p^d_\kk)$, which is in particular Zariski continuous. If $\kk$ is a local field then $\hat{\rho}$ is Euclidean continuous by Lemma~\ref{lem:morph_eucl_cont}.

Now, suppose that $\rho$ extends according to our hypothesis and $\hat{\rho}_0\in\aut(\p^d_\kk)$. In particular, the limit of $F_i^t$ for $t\to 0$ is a rational map for all $i=1,\dots,d$, which implies that $P_{i0}=0$ and $Q_{i0}\neq0$.
Hence $f$ fixes $[1:0:\cdots:0]$, and $\hat{\rho}_0$ corresponds to the derivative of $f$ at $[1:0:\cdots:0]$, so $f$ is a local isomorphism at $p$.
\end{proof}

\begin{proposition}\label{pro:determined_by_aut}
Let $\kk$ be an infinite field and $\varphi\colon\Cr_d(\kk)\to\Cr_d(\kk)$ a group homomorphism which is Zariski continuous. If $\varphi|_{\aut(\p^d_\kk)}=\id_{\aut(\p^d_\kk)}$, then $\varphi=\id_{\Cr_d(\kk)}$. 

The same statement holds for any local field $\kk$ and $\varphi$ Euclidean continuous.
\end{proposition}

\begin{proof}
Let $f\in\Cr_d(\kk)$ and $p\in\p^d_{\kk}(\kk)$ any point such that both $f$ and $\varphi(f)$ are local isomorphisms at $p$. We will prove that  $\varphi(f)(p)=f(p)$ or $\varphi(f)(p)=p$, which implies that $\varphi(f)=f$ or $\varphi(\id)$, since $\kk$ is infinite. From this we will then deduce the claim.

We denote $q:=f(p)\in\p^d_{\kk}(\kk)$. There exists an automorphism $\alpha\in\aut(\p^d_\kk)$ whose only fixed points in $\p^d_{\kk}(\kk)$ are $p$ and $q$. 
Let us observe that the birational transformation $\alpha^{-1}f^{-1}\alpha f$ fixes $p$ and is a local isomorphism at $p$. 
Denote by $\rho\colon\A_\kk^1\setminus\{0\}\to\Cr_d(\kk)$ the morphism defined in Example~\ref{ex:deformation} with respect to the map $\alpha^{-1}f^{-1}\alpha f$, i.e. $\rho_t=\beta_t^{-1}( \alpha^{-1}f^{-1}\alpha f)\beta_t$. 
By Lemma~\ref{lem:deformation_exists}, $\rho$ extends to a morphism $\hat{\rho}\colon\A^1\to\Cr_d(\kk)$ and $\hat{\rho}_0$ is the derivative of $(\alpha^{-1}f^{-1}\alpha f)$ at $p$. 
Lemma~\ref{lem:morph_eucl_cont} yields that the corresponding map $\A_\kk^1(\kk)\to\Cr_d(\kk)$ is also Euclidean continuous. 

Since $\varphi$ is Zariski (respectively Euclidean) continuous, the map $\varphi\circ\hat{\rho}\colon\A^1(\kk)\to\Cr_d(\kk)$ is also Zariski (respectively Euclidean) continuous. By hypothesis, the restriction of $\varphi$ to $\aut(\p^d_\kk)$ is the identity, and so
\begin{align*}
&\varphi(\rho_t)=\beta_t^{-1} \varphi(\alpha^{-1}f^{-1}\alpha f)\beta_t=\beta_t^{-1} \alpha^{-1}\varphi(f)^{-1}\alpha\varphi(f)\beta_t,\quad t\neq0\\
&\varphi(\hat{\rho}_0)=\hat{\rho}_0 \in\aut(\p^d_\kk)
\end{align*}
Lemma~\ref{lem:deformation_exists} implies that the map $(\alpha^{-1}\varphi(f)^{-1}\alpha\varphi(f))$ fixes $p$ and is a local isomorphism at $p$. In particular, since $\varphi(f)$ is a local isomorphism at $p$, we have
\[\alpha(\varphi(f)(p))=\varphi(f)(\alpha(p))=\varphi(f)(p).\]
By hypothesis, $p$ and $q=f(p)$ are the only fixed points of $\alpha$, hence 
\[\varphi(f)(p)=p\quad\text{or}\quad\varphi(f)(p)=f(p).\]
This holds for any $p$ in an open dense subset of $\p^d_{\kk}(\kk)$. 
As $\kk$ is infinite, $\varphi(f)$ coincides thus with $\id_{\p^d_\kk}$ or with $f$. This holds for any $f\in\Cr_d(\kk)$, so in particular for $\gamma f$, for any non-trivial $\gamma\in\aut(\p^d_\kk)$.
Suppose that $\varphi(f)=\id$. Then $\id\neq\gamma=\varphi(\gamma)=\varphi(\gamma f)$. Since we have by  the same argument again that either $\varphi(\gamma f)=\gamma f$ or $\varphi(\gamma f)=\id$, we obtain that  $\varphi(\gamma f)=\gamma f$, which implies $f=\id$.  
It follows that $\varphi(f)=f$ for any $f\in\Cr_d(\kk)$.
\end{proof}

We can repeat the proof of Proposition~\ref{pro:determined_by_aut} word by word by considering $\aut(\A^d_\kk)$ instead of $\Cr_d(\kk)$ and its subgroup $\aff_d(\kk)$ instead of $\PGL_{d+1}(\kk)$. This way, we obtain the following statement:

\begin{proposition}\label{pro:determined_by_aff}
Let $\kk$ be an infinite field and $\varphi\colon\aut(\A^d_\kk)\to\aut(\A^d_\kk)$ a surjective homomorphism which is  Zariski continuous. If $\varphi|_{\aff_d(\kk)}=\id_{\aff_d(\kk)}$, then $\varphi=\id_{\aut(\A^d_\kk)}$. 
\end{proposition}

\section{Proofs of the main theorems}

Group automorphisms of classical groups are well-understood. In \cite{MR0310083} J. Dieu\-donn\'e gives a complete classification in the case of fields, 
and we need the classification of group automorphisms of $\PGL_{d+1}(\kk)$.

\begin{theorem}[{\cite[IV.\S1.I--III, p.85--89 and IV.\S6, p.98]{MR0310083}}]\label{thm:dieudonne}
Let $d\geq1$ and $\kk$ a field. 
Let $G$ be $\PGL_{d+1}(\kk)$ or $\GL_{d+1}(\kk)$.
For any group automorphism $\varphi\colon G\to G$ there exists an element $h\in G$ and a field automorphism $\alpha$ of $\kk$ such that $\varphi$ is of the form
\[\varphi(g)=h({}^{\alpha}\!g^{\vee})h^{-1}\quad\text{or}\quad \varphi(g)= h({}^{\alpha}\!g)h^{-1}\]
where $g^{\vee}={}^t(g^{-1})$ is the inverse of the transpose of $g$.
\end{theorem}

Not all automorphisms of $\PGL_{d+1}(\kk)$ can be extended to an automorphism of the Cremona group:

\begin{lemma}\cite[Corollary A.12]{longversion} 
\label{lem:extension}
For any field $\kk$, 
	the automorphism of $\PGL_{d+1}(\kk)$ given by $g\mapsto {}^\alpha\!g^\vee$ does not extend to an automorphism of $\Cr_d(\kk)$. 
\end{lemma}

\begin{theorem}[{\cite[Theorem A, Corollary 8.5]{cantat2014algebraic}}]\label{thm:cantatxie}
	Let $d\geq 1$ and $\Gamma\subset\SL_{d+1}(\Z)$ a finite index subgroup. Let $X_{\C}$ be an irreducible complex quasi-projective variety of dimension $\dim(X)=n$. 
	\begin{enumerate}
		\item If there is an injective group homomorphism  $\varphi\colon\Gamma\to \aut(X_{\C})$, then \hbox{$n\geq d$.} \\
		If $n=d$ then there is an isomorphism $f\colon X_{\C}\to \p^{d}_{\C}$ such that $f\varphi(\Gamma) f^{-1}\subset\aut(\p^d_{\C})$.
	
		\item If there is an injective group homomorphism  $\psi\colon\Gamma\to \Bir(X_{\C})$, then $n\geq d$. \\
		If $n=d$ then there is a birational transformation $f\colon X_{\C}\dashrightarrow \p^d_{\C}$ such that $f\varphi(\Gamma) f^{-1}\subset\aut(\p^d_{\C})$.
	\end{enumerate}
\end{theorem}

Let $p$ be an odd prime. We denote by $\Gamma_p\subset\SL_{d+1}(\Z)$ the congruence subgroup mod $p$, i.e. the kernel of the homomorphism $\rho\colon\SL_{d+1}(\Z)\to\SL_{d+1}(\F_p)$ given by reduction modulo $p$. Since $p$ is odd by assumption, $\Gamma_p$ intersects the center of $\SL_{d+1}(\Z)$ only in the identity and hence, the restriction of the quotient homomorphism $\GL_{d+1}(\kk)\to\PGL_{d+1}(\kk)$ to $\Gamma_p$ is injective and we can consider $\Gamma_p$ as a subgroup of $\PGL_{d+1}(\kk)$. We can now apply Theorem~\ref{thm:cantatxie} to $\Gamma_p$ and obtain the following:
\vskip\baselineskip

\begin{corollary}\label{cor:cantatxie}
Let $d\geq2$, $\kk$ a field of characteristic $0$ and $\varphi\colon\Gamma_p\to\Cr_d(\kk)$ an injective group homomorphism. Then the degree of the elements $\varphi(\Gamma_p)$ is bounded.
\end{corollary}

\begin{proof}
	Let $F\subset\kk$ be the smallest subfield over which $\varphi(\Gamma_p)$ is defined.  We can consider $\varphi(G)$ as a subgroup of $\Cr_d(F)\subset\Cr_d(\kk)$. Since $\Gamma_p$ is finitely generated, $F$ is a finitely generated extension of $\Q$ and as such embeds into $\CC$ and we may consider $\varphi(G)$ as a subgroup of $\Cr_d(F)\subset\Cr_d(\CC)$. By Theorem~\ref{thm:cantatxie}, $\varphi(G)$ is conjugate in $\Cr_d(\CC)$ to a subgroup of $\aut(\p^d_{\C})$. In particular, $\varphi(\Gamma_p)$ is of bounded degree. 
\end{proof}

Now, we study continuous embeddings of $\PGL_{d+1}(\kk)$ into $\Cr_d(\kk)$. 

\begin{lemma}\label{lem:PSL}
	Let  $\kk$ be either the field of real numbers $\R$ or the field of complex numbers $\CC$ and $d\geq 1$.  Let $\varphi\colon\PSL_{d+1}(\kk)\to\Cr_d(\kk)$ be an injective group homomorphism that is Euclidean continuous. Then there exists an element $f\in\Cr_d(\kk)$ such that the group $f\varphi(\PSL_{d+1}(\kk))f^{-1}$ is contained in $\aut(\p^d_{\kk})$.
\end{lemma}

\begin{proof}
	The statement is trivial for $d=1$, hence we can assume that $d\geq2$.

	Let $p$ be an odd prime and $\Gamma_p\subset\SL_{d+1}(\Z)$ the congruence subgroup mod $p$. We consider $\Gamma_p$ as a subgroup of $\PGL_{d+1}(\kk)$. In particular, $\Gamma_p$ contains the subgroup of elements of the form 
	\[[x_0:\dots:x_d]\mapsto[x_0+kpx_1:x_1:\dots: x_d], \quad k\in\Z.\]
	
	For $i,j=0,\dots,d$, $i\neq j$, define the elementary subgroups
	\[
	U_{ij}:=\left\{[x_0:\dots:x_d]\mapsto[x_0:\dots:x_i+cx_j:x_{i+1}:\dots: x_d]\mid c\in\kk\right\}\subset\PGL_{d+1}(\kk).
	\]
	
	Consider the Euclidean compact subset 
	$$E_{01}:=\{[x_0:\dots:x_d]\mapsto[x_0+cx_1:x_1:\dots: x_d]\mid ||c||\leq p\}\subset\PGL_{d+1}(\kk)$$ where $||\cdot||$ denotes the Euclidean norm on $\kk$. Every element in the subgroup $U_{01}$ can be written as a product of an element in $\Gamma_p$ and an element in $E_{01}$. 
	Since $E_{01}$ is compact, $\varphi(E_{01})$ is of bounded degree \cite[Lemma 5.13]{BlancFurter}.
	By Corollary~\ref{cor:cantatxie}, the image $\varphi(\Gamma_p)$ is of bounded degree as well and we conclude that $\varphi(U_{00})$ is of bounded degree. Since all the elementary subgroups $U_{ij}$ are conjugate to $U_{00}$ in $\PGL_{d+1}(\kk)$ it follows that $\varphi(U_{ij})$ is of bounded degree for any $i,j$.
	
	Using the Gauss-Algorithm one can see that there exists an integer $K\geq1$ only depending on $d$ such that every element in $\PSL_{d+1}(\kk)$ can be written as the product of at most $K$ elements contained in some the subgroups $U_{ij}$. Therefore, $\varphi(\PSL_{d+1}(\kk))$ is of bounded degree and hence is regularisable, i.e. there exists a quasi-projective $\kk$-variety $X_{\kk}$ and a birational transformation $ \p^d_{\kk}\dashrightarrow X_{\kk}$ that conjugates $\varphi(\PSL_{d+1}(\kk))$ to a subgroup of $\aut(X_{\kk})$ \cite[Theorem 1]{Rosenlicht}\footnote{\label{note} In fact, we apply \cite[Theorem 1]{Rosenlicht} to the Zariski closure of $\varphi(\PSL_{d+1}(\kk))$ in $\Cr_d(\kk)$, which is an algebraic subgroup of $\Cr_d(\kk)$.}.
	By Theorem~\ref{thm:cantatxie}, we have $X_{\C}\simeq\p^d_{\C}$ and 
	hence, by Ch\^atelet's theorem, $X_{\kk}\simeq \p^d_{\kk}$ \cite[§IV.I, p.283]{chatelet}. 
\end{proof}

\begin{proposition}\label{prop:emb_PGL}
	Let $\kk$ be a field of characteristic $0$ and let $\varphi\colon \PGL_{d+1}(\kk)\to\Cr_d(\kk)$ be an injective group homomorphism. 
	 If either $\varphi$ is Zariski continuous, or  if $\kk=\R$ or $\kk=\C$ and $\varphi$ is Euclidean continuous,  
	then $\varphi(\PGL_{d+1}(\kk))$ is conjugate in $\Cr_d(\kk)$ to a subgroup of $\aut(\p^d_{\kk})$.
\end{proposition}

\begin{proof}
	If $\varphi$ is Zariski continuous, then $\varphi(\PGL_{d+1}(\kk))$  is of bounded degree by Lemma~\ref{lem:bounded_degree}. 
	If $\kk=\R$ or $\kk=\C$ and $\varphi$ is Euclidean continuous,  then Lemma~\ref{lem:PSL} implies that $\varphi(\PSL_{d+1}(\kk))$ has bounded degree in $\Cr_d(\kk)$. Since $\PSL_{d+1}(\kk)$ has finite index in $\PGL_{d+1}(\kk)$, hence also $\varphi(\PGL_{d+1}(\kk))$ is of bounded degree.
	
	In either case there exists a quasi-projective $\kk$-variety $X_{\kk}$ and a birational map $\p^d_{\kk}\rat X_\kk$ which conjugates $\varphi(\PGL_{d+1}(\kk))$ to a subgroup of $\aut(X_{\kk})$ \cite[Theorem 1]{Rosenlicht}\footref{note}. 
	By Theorem~\ref{thm:cantatxie}, $X_{\overline{\kk}}\simeq\p^d_{\overline{\kk}}$, where $\overline{\kk}$ is the algebraic closure of $\kk$, and hence, by Ch\^atelet's theorem, $X_{\kk}\simeq\p^d_{\kk}$ \cite[§IV.I, p.283]{chatelet}. 
\end{proof}

\begin{proof}[Proof of Theorem~\ref{main:cremona} and Theorem~\ref{main:euclidean}]
The group automorphism  $\varphi\colon\Cr_d(\kk)\to\Cr_d(\kk)$ induces an injective group  homomorphism $\varphi|_{\aut(\p^d_{\kk})}\colon \aut(\p^d_{\kk})\hookrightarrow\Cr_d(\kk)$.
	Since $\varphi$ is continuous, Proposition~\ref{prop:emb_PGL} implies that up to conjugation in $\Cr_d(\kk)$, we may assume that $\varphi(\aut(\p^d_{\kk}))\subset\aut(\p^d_{\kk})$. 
	As by assumption $\varphi^{-1}$ is a continuous group homomorphism as well, we argue analogously that there exists $f\in\Cr_d(\kk)$ such that $f\varphi^{-1}(\aut(\p^d_{\kk}))f^{-1}\subset\aut(\p^d_{\kk})$. 
	In particular, we have $f\aut(\p^d_{\kk})f^{-1}\subset\aut(\p^d_{\kk})$ and so
	 $f$ normalizes $\aut(\p^d_{\kk})$ and hence $f\in\aut(\p^d_{\kk})$, i.e. $\varphi^{-1}(\aut(\p^d_{\kk}))=\aut(\p^d_{\kk})$.
	With Theorem~\ref{thm:dieudonne} and Lemma~\ref{lem:extension} we conclude that, up to conjuation in $\aut(\p^d_{\kk})$ and up to a field automorphism of $\kk$, we have $\varphi|_{\aut(\p^d_{\kk})}=\id_{\aut(\p^d_{\kk})}$.
	Since $\varphi$ is continuous, Proposition~\ref{pro:determined_by_aut} implies that $\varphi=\id_{\Cr_d(\kk)}$, which yields the claim.

\end{proof}

%%%%%%%%%%%%%%%%%%%%%%%%

\section{Polynomial automorphisms}\label{sec:poly}
Denote by $\aut(\A_\kk^d)$ the {\it group of polynomial automorphisms} of the affine $d$-space $\A_\kk^d$ over a field $\kk$. The goal of this section is to prove Theorem~\ref{main:aut}. If we choose an embedding of $\A_\kk^d$ into $\p_\kk^d$, we can naturally consider $\aut(\A_\kk^d)$ as a subgroup of $\Cr_d(\kk)$.  The {\it Zariski topology on $\aut(\A_\kk^d)$} is the induced topology of the Zariski topology on $\Cr_d(\kk)$. 

It can be shown that if $\kk$ is algebraically closed, then $\aut(\A_\kk^d)$ can be naturally equipped with the additional structure of a so-called {\it ind-group}, which induces the {\it ind-topology} on $\aut(\A_\kk^d)$. We refer to \cite{furterkraft} for a definition and details about this notion. We will only need that an ind-closed subgroup of bounded degree of $\aut(\A_\kk^d)$ has the structure of an affine algebraic group such that the induced action on $\A_\kk^d$ is algebraic \cite[Proposition~3.7]{stampfli2013contributions}. 

Denote by $\T_d(\kk)\subset\aut(\A^d_\kk)$ the diagonal linear transformations. 
Let $\ak$ be the algebraic closure of a field $\kk$. In what follows, we always consider $\aut(\A_\kk^d)$ as a subgroup of $\aut(\A^d_\ak)$ through the obvious inclusion.

\begin{lemma}
\label{lem:torus}
Let $\kk$ be an infinite field and $\varphi\colon\aut(\A^d_{\kk})\to\aut(\A^d_{\kk})$ a group automorphism that is a homeomorphism with respect to the Zariski topology. 
Then there exists $f\in\aut(\A^d_{\ak})$ such that $f\varphi(T_d(\kk))f^{-1}\subset T_d(\ak)$ is a dense subgroup.
\end{lemma}

\begin{proof}
	Let $D_{\ak}$ be the  closure  of $\varphi(\T_d(\kk))$ in $\aut(\A^d_{\ak})$ with respect to the ind-topology. Since $\varphi(\T_d(\kk))$ and hence  $D_{\ak}$ are of bounded degree, we obtain that $D_{\ak}$ is an affine algebraic group acting algebraically on $\A^d_\ak$. By the structure theorem for commutative affine algebraic groups over algebraically closed fields (see \cite[Théorème~XVII.7.2.1]{MR0274459}), $D_{\ak}$ is isomorphic over $\ak$ to $T_r(\ak)\times U_s(\ak)$, where $T_r$ is a split torus of dimension $r\geq0$ and $U_s$ a unipotent group of dimension $s\geq0$.  
	
	First, suppose that $\car(\kk)\neq 2$. 
	The image $\varphi(T_d(\kk))$ contains $2^d$ elements of order two, hence  $D_{\ak}$ contains at least $2^d$ elements of order two, and it follows that $r\geq d$.
	Thus $D_{\ak}$ contains a maximal torus in $\aut(\A_{\ak}^d)$ and as such coincides with its centraliser. So we obtain that $s=0$.  
	
	Now, if $\car(\kk)=2$, as $\varphi$ is a homeomorphism with respect to the Zariski topology by hypothesis, and since $T_d(\kk)^2$ is Zariski dense in $T_d(\kk)$, we have that $\varphi(T_d(\kk))^2$ is dense in $\varphi(T_d(\kk))$ and hence in $\overline{\varphi(T_d(\kk))}$. This implies in particular that $\overline{\varphi(T_d(\kk))}^2$ is dense in $\overline{\varphi(T_d(\kk))}$, i.e. $(D_\ak)^2$ is dense in $D_\ak$.
	Now, at the same time, $\car(\kk)=2$ implies that $(D_{\ak})^2\simeq T_r(\ak)$, and we obtain  $s=0$. Since $\varphi$ is a homeomorphism, it preserves the Krull dimension, therefore $\varphi(\T_d(\kk))$ and hence also $D_{\ak}$ has Krull dimension $d$, it follows also that $d=r$.
	
	So we have that $D_{\ak}$ is a torus of rank $d$. From
	\cite[Theorem 2]{MR0200279} it follows that $D_{\ak}$ is conjugate in $\aut(\A_{\ak}^d)$ to the standard torus $T_d(\ak)$ in $\aut(\A_{\ak}^d)$. In particular, there exists an element $f\in\aut(\A_{\ak}^d)$ such that $f\varphi(T_d(\kk))f^{-1}\subset T_d(\ak)$. Since $\varphi(T_d(\kk))$ is dense in $D_{\ak}$, we obtain that $f\varphi(T_d(\kk))f^{-1}$ is dense in $T_d(\ak)$. 
\end{proof}

Let us denote by $\GL_d(\kk)\subset \aut(\A^d)$ the subgroup of linear automorphisms, by $\aff_d(\kk)$ the subgroup of affine transformations, and by $\s_d\subset\GL_d(\kk)$ the subgroup of coordinate permutations. 
For a group $G$ we denote by  $\zenter(G)$ the center of $G$, and for a set $A\subset G$ we denote by $\cent_G(A)$  and by $\norm_G(A)$ the centraliser and normaliser of $A$ in $G$, respectively.

We leave it as an exercise to the reader to verify that for an infinite field $\kk$ the centralizer of $\zenter(\GL_d(\kk))$ in $\aut(\A_\kk^d)$  is $\GL_d(\kk)$, that $\norm_{\aut(\A^d_\kk)}(T_d(\kk)) =\T_d(\kk)\rtimes \s_d$ and $C(\T_d(\kk)\rtimes\s_d))=\zenter(\GL_d(\kk))$.

\begin{lemma}
\label{lem:GL}
Let $\kk$ be an infinite field and $\varphi\colon\aut(\A^d_{\kk})\to\aut(\A^d_{\ak})$ an injective group endomorphism. Assume that $\varphi(T_d(\kk))$ is contained in $T_d(\ak)$ as a dense subgroup. 
Then $\varphi(\GL_d(\kk))$ is contained in $\GL_d(\ak)$. Moreover, this image is dense.
\end{lemma}

\begin{proof}
Since $\varphi(T_d(\kk))$ is dense in $T_d(\ak)$ by assumption, we have 
\[  
\norm_{\aut(\A^d_\ak)}(\varphi(T_d(\kk)))=\norm_{\aut(\A^d_\ak)}(T_d(\ak))=T_d(\ak)\rtimes\s_d.
\]
As the image of the normaliser is contained in the normaliser of the image, this implies that $\varphi(\s_d)\subset T_d(\ak)\rtimes\s_d$.
Consider the natural projection $\pi\colon\T_d(\ak)\rtimes\s_d\to\s_d$. Its kernel is equal to $\T_d(\ak)$, which is equal to $\cent_{\aut(\A^d_\ak)}(T_d(\kk))$. 
Then
\[
\varphi(\s_d)\cap\T_d(\ak)=\varphi(\s_d)\cap\cent_{\aut(\A^d_\ak)}(\T_d(\ak))=\varphi(\s_d\cap\T_d(\ak))=\{\id\}
\] 
implies that up to an automorphism of $\s_d$ the restriction $\varphi|_{\s_d}$ is a section of $\pi$.
Therefore, $\cent_{\T_d(\ak)\rtimes\s_d}(\varphi(\s_d))=\zenter(\GL_d(\ak))$.
It follows that
\[
\varphi(\zenter(\GL_d(\kk))=\varphi\left(\cent_{\T_d(\kk)\rtimes\s_d}(\s_d)\right)\subset\cent_{\T_d(\ak)\rtimes\s_d}(\varphi(\s_d))=\zenter(\GL_d(\ak)).
\]
As $\kk$ is infinite and $\varphi$ is injective, the image $\varphi(\zenter(\GL_d(\kk)))$ is dense in $\zenter(\GL_d(\ak))$. 
It follows that
\begin{align*}
\varphi(\GL_d(\kk))&=\varphi\left(\cent_{\aut(\A^d_{\kk})}(\zenter(\GL_d(\kk)))\right)\\
&\subseteq\cent_{\aut(\A^d_{\ak})}(\varphi(\zenter(\GL_d(\kk))))\\
&=\cent_{\aut(\A^d_{\ak})}(\zenter(\GL_d(\ak)))=\GL_d(\ak).
\end{align*}
It remains to show that $\varphi(\GL_d(\kk))$ is dense in $\GL_d(\ak)$.
For $i,j=1,\dots,d$, $i\neq j$, we define the elementary subgroups
	\[
	U_{ij}(\ak):=\left\{(x_1,\dots,x_d)\mapsto(x_1,\dots,x_i+cx_j,x_{i+1},\dots,x_d)\mid c\in\ak\right\}\subset\GL_d(\ak).
	\]
	and
	\[
	U_{ij}(\ak)\ni E_{ij}\colon (x_1,\dots,x_d)\mapsto(x_1,\dots,x_i+x_j,x_{i+1},\dots,x_d)
	\]
	Observe that 
	\[\{tE_{ij}t^{-1}\mid t\in\varphi(T_d(\kk))\}\subseteq \{tE_{ij}t^{-1}\mid t\in T_d(\ak)\}=U_{ij}(\ak) \]
	is a dense subset which is also contained in $\varphi(\GL_d(\kk))$. 
	Recall that there exists an integer $N\geq 1$ and a surjective morphism $V_1\times\dots\times V_N\to\GL_d(\ak)$, $(v_1,\dots,v_N)\mapsto v_1v_2\cdots v_N$, where each $V_l$ is equal to $T_d(\ak)$ or to $U_{ij}(\ak)$ for some $i,j$. 
	Since $V_l':=V_l\cap\varphi(\GL_d(\kk))$ is dense in $V_l$ as remarked above, this implies that $\varphi(\GL_d(\kk))$ is dense in $\GL_d(\ak)$.
\end{proof}

Let $G$ be a group acting on a group $A$ by $\sigma\colon a\mapsto {}^\sigma\! a$ for $\sigma\in G, a\in A$. Recall that the first Galois cohomology $H^1(G,A)$ is defined as the set of {\it cocycles}, i.e. maps $\nu\colon G\to A$ satisfying $\nu(\sigma\tau)=\nu(\sigma)\ {}^\sigma\!\nu(\tau)$ for all $\sigma, \tau\in G$, up to the following equivalence relation: two cocycles $\nu$ and $\eta$ from $G$ to $A$ are equivalent if there exists an element $a\in A$ such that $\eta(\sigma)=a^{-1}\nu(\sigma)\ {}^\sigma\!a$ for all $\sigma\in G$. 

\begin{lemma}
\label{lem:Hilbert}
	Let $\kk$ be a perfect field and $\varphi\colon \GL_d(\kk)\to\aut(\A^d_\kk)$ an injective homomorphism. Assume that there is an element $f\in\aut(\A^d_{\ak})$ such that $f\varphi(\GL_d(\kk))f^{-1}\subset \GL_d(\ak)$ as a dense subgroup, then there exists an element $g\in\aut(\A^d_{\kk})$ such that $g\varphi(\GL_d(\kk))g^{-1}\subset \GL_d(\kk)$.
\end{lemma}

\begin{proof}
	Let $L$ be the finite field extension such that $f$ is defined over $L$. In particular, $H=f\varphi(\GL_d(\kk))f^{-1}\subset\GL_d(L)\subset\aut(\A^d_L)$ as a dense subgroup. Let $\sigma\in G=\Gal(L/\kk)$ be any element of the Galois group of $L$ over $\kk$. Then 
	\[
	{}^\sigma\!H={}^\sigma\!(f\varphi(\GL_d(\kk))f^{-1})={}^\sigma\!f(\varphi(\GL_d(\kk)){}^\sigma\! f^{-1}={}^\sigma\!ff^{-1}Hf({}^\sigma\!f)^{-1}.
	\]
	 Using that $H$ and hence also ${}^\sigma\! H$ is dense in $\GL_d(L)$, we may conclude that ${}^\sigma\!f f^{-1}$ normalizes $\GL_d(L)$ and therefore that ${}^\sigma\!f f^{-1}\in\GL_d(L)$. In other words, there exists an element $g_\sigma\in \GL_d(L)$ such that ${}^\sigma\!f =g_\sigma f$. 
	
	We define now the map $\mu\colon G\to\GL_d(L)$ by $\mu(\sigma):=g_\sigma^{-1}$. It is straightforward to check that $\mu$ satisfies the cocycle condition with respect to the Galois action of $G$ on $\GL_d(L)$. By Hilbert 90 (see for example \cite[Lemma~1, Chapter~III.1.1]{MR1466966}) the first cohomology $H^1(G,\GL_d(\kk))$ is trivial, i.e. every cocycle is  equivalent to the cocycle that is given by mapping every element in $G$ to the identity in $\GL_d(L)$. In other words, there exists an element $a\in\GL_d(L)$ satisfying 
	$a^{-1}\mu(\sigma) {}^\sigma\! a=\id$ for all $\sigma\in G$. Consider now the polynomial automorphism $h=a^{-1}f$. Clearly we 
	have $h\varphi(\GL_d(L))h^{-1}\subset \GL_d(L)$. It remains to show that $h$ is defined over $\kk$. Since $\kk$ is perfect, it is equivalent to prove ${}^\sigma\! h=h$ for all $\sigma\in G$. For any $\sigma\in G$ one calculates ${}^\sigma\! h={}^\sigma\! a^{-1}{}^\sigma\! f=a^{-1}\mu(\sigma){}^\sigma\! f=a^{-1}\mu(\sigma)\mu(\sigma)^{-1} f=h$. As $h$ and $\varphi(\GL_d(\kk))$ are both defined over $\kk$, the group $h\varphi(\GL_d(\kk))h^{-1}$ is defined over $\kk$ as well, so $h\varphi(\GL_d(\kk))h^{-1}\subset\GL_d(\kk)$.
\end{proof}

\begin{lemma}
\label{lem:GLtoGL}
	Let $\kk$ be an infinite perfect field and $\varphi\colon \aut(\A^d_\kk)\to\aut(\A^d_\kk)$ a group automorphism that is a homeomorphism with respect to the Zariski topology. Then there exists an element $f\in\aut(\A^d_\kk)$ such that $f\varphi(\GL_d(\kk))f^{-1}=\GL_d(\kk)$. 
\end{lemma}

\begin{proof}
From Lemma~\ref{lem:torus}, Lemma~\ref{lem:GL} and Lemma~\ref{lem:Hilbert} we obtain that there exists $f\in\aut(\A^d_\kk)$ such that $f\varphi(\GL_d(\kk))f^{-1}\subseteq\GL_d(\kk)$. Since, by assumption, $\varphi$ is a homeomorphism, Lemma~\ref{lem:bounded_degree} implies that $\varphi^{-1}(f^{-1}\GL_d(\kk)f)$ is of bounded degree and closed, and is hence an affine algebraic group containing $\GL_d(\kk)$. The Krull-dimension is preserved by a homeomorphism, thus $\GL_d(\kk)=\varphi^{-1}(f^{-1}\GL_d(\kk)f)$.
\end{proof}

\begin{lemma}
\label{lem:aff_id}
Let $\kk$ be an infinite perfect field and $\varphi\colon \aut(\A^d_\kk)\to\aut(\A^d_\kk)$ a group automorphism and suppose that $\varphi(\GL_d(\kk))=\GL_d(\kk)$.  
Then $\varphi(\aff_d(\kk))=\aff_d(\kk)$ and there exists a field automorphism $\alpha$ of $\kk$ and an element $h\in\GL_d(\kk)$ such that $\varphi(g)=h({}^{\alpha}\!g)h^{-1}$ for all $g\in\aff_d(\kk)$.
\end{lemma}

\begin{proof}
By Theorem~\ref{thm:dieudonne}, up to conjugation in $\GL_d(\kk)$ and up to a field automorphism of $\kk$, we can assume that $\varphi|_{\GL_d(\kk)}$ is the identity on $\GL_d(\kk)$ or given by $g\mapsto g^{\vee}$. 

Let $f\colon(x_1,\dots,x_d)\mapsto (x_1+a,x_2,\dots,x_d)$ for some $a\in\kk^*$. 
Consider $\GL_{d-1}(\kk)$ as a subgroup of $\GL_d(\kk)$ embedded as $g\mapsto[(x_1,\dots,x_d)\mapsto(x_1,g(x_2,\dots,x_d))]$. 
We write $\varphi(f)\colon x\mapsto(p_1(x),\dots,p_d(x))$ for $x=(x_1,\dots,x_d)$ and polynomials $p_1,\dots,p_d\in\kk[x_1,\dots,x_d]$.
Observe that $f$ commutes with every element of $\GL_{d-1}(\kk)$.
As $\varphi(\GL_{d-1}(\kk))=\GL_{d-1}(\kk)$, it follows that the same is true for $\varphi(f)$, and hence $p_i(x)=x_ib$ for some $b\in\kk^*$ for $i=2,\dots,d$.
Moreover, $p_1(x)$ is of the form $cx_1+d$ for some $c,d\in\kk$. Note that $d\neq0$ because $\varphi$ is injective. It follows in particular, that $\varphi(\aff_d(\kk))=\aff_d(\kk)$.

Consider $h\colon(x_1,\dots,x_d)\mapsto(x_1+x_2,x_2,x_3,\dots,x_d)$. Then $f$ commutes with $h$ but $\varphi(f)$ does not commute with $h^{\vee}$. It follows that $\varphi|_{\GL_d(\kk)}$ is the identity on $\GL_d(\kk)$. Then $\varphi(f)h=h\varphi(f)$, which implies that $d=b$.

If $\car(\kk)=2$, then $f^2=\id$ and hence $\varphi(f)^2=\id$, which implies that $b=1$. 
Now, suppose that $\car(\kk)\neq2$ and consider $t\colon(x_1,\dots,x_d)\mapsto(2x_1,x_2,\dots,x_d)$. We have $tft^{-1}=f^2$, so also $t\varphi(f)t^{-1}=\varphi(f)^2$, which again implies $b=1$.

After conjugating $\varphi$ with a suitable element of the center of $\GL_d(\kk)$, we obtain $\varphi(f)=f$. 
The group $\GL_d(\kk)$ acts transitively by conjugation on the set of translations different from $\id$, therefore $\varphi|_{\aff_d(\kk)}=\id_{\aff_d(\kk)}$.
\end{proof}

\begin{proof}[Proof of Theorem~\ref{main:aut}]
The theorem follows from Lemma~\ref{lem:GLtoGL}, Lemma~\ref{lem:aff_id} and
Proposition~\ref{pro:determined_by_aff}.
\end{proof}

%%%%%%%%%%%%%%%%%%%%%%%%
\section{Birational diffeomorphisms}\label{sec:birdiff}

Let $X_\R$ be a real projective rational variety, such that the set of $\R$-points $X_\R(\R)$ is non-empty. We can now consider the subgroup $\BirDiff({X_\R(\R)})\subset\Bir(X)$ of {\it birational diffeomorphisms}, which is the subgroup of birational transformations without any real indeterminacy points. In \cite{MR2531367} it is shown that $\BirDiff({\p_\R^2(\R)})$ and $\BirDiff({\p_\R^1(\R)\times\p_\R^1(\R)})$ are dense subgroups of the diffeomorphism groups $\mathcal{C}^\infty(\p^2(\R))$ and $\mathcal{C}^\infty({\p_\R^1(\R)\times\p_\R^1(\R)})$ respectively, with respect to the $\mathcal{C}^\infty$-topology.

The Zariski (resp. Euclidean) topology on $\Cr_d(\R)$ induces the {\em Zariski} (resp. {\em Euclidean}) topology on $\BirDiff(\p_\R^d(\R))$.

\begin{remark}
\label{rmk:approx_R}
The proofs of Lemma~\ref{lem:deformation_exists} and Proposition~\ref{pro:determined_by_aut} can be repeated word by word for $\BirDiff(\p_\R^d(\R))$ instead of $\Cr_d(\R)$. We obtain that any 
group homomorphism of $\BirDiff(\p_\R^d(\R))$ that is Zariski or Euclidean continuous is uniquely determined by its restriction to $\PGL_{d+1}(\R)$.
\end{remark}

\begin{proposition}\label{main:birdiff}
	Let $\varphi\colon\BirDiff(\p_\R^d(\R))\to\BirDiff(\p_\R^d(\R))$ be a group automorphism  that is a homeomorphism with respect to the Zariski topology on the group $\BirDiff(\p_\R^d(\R))$. 
	Then either $\varphi$ is an inner automorphism, or 
	there exists $f\in\BirDiff(\p_\R^d(\R))$ such that $\varphi(g)=fg^{\vee}f^{-1}$ for all $g\in\PGL_{d+1}(\R)$ and this defines $\varphi$ uniquely. 
\end{proposition}

\begin{proof}
By Proposition~\ref{prop:emb_PGL}, there exists $f\in\Cr_d(\R)$ such that $f\varphi(\PGL_{d+1}(\R))f^{-1}\subset\PGL_{d+1}(\R)$. 
Since $\PGL_{d+1}(\R)$ acts transitively on $\p_\R^d(\R)$, the maps $f$ and $f^{-1}$ have no base-points in $\p_\R^d(\R)$, and hence $f\in\BirDiff(\p_\R^d(\R))$. 
So, we can assume that $\varphi(\PGL_{d+1}(\R))=\PGL_{d+1}(\R)$. 
By Theorem~\ref{thm:dieudonne}, up to conjugation in $\PGL_{d+1}(\R)$, we have that $\varphi|_{\PGL_{d+1}(\R)}=\id_{\PGL_{d+1}(\R)}$ or $\varphi(g)=g^{\vee}$ for all $g\in\PGL_{d+1}(\R)$. 
By Remark~\ref{rmk:approx_R}, $\varphi$ is uniquely determined by its restriction on $\PGL_{d+1}(\R)$.
\end{proof}

Despite best efforts, the authors were not able to determine whether the automorphism $\PGL_{d+1}(\R)$ given by $g\mapsto {}^\alpha\!g^\vee$ extends to an endomorphism of $\BirDiff(\p_\R^d(\R))$.

\begin{remark}
For any field $\kk$, we can define the group $\mathrm{Bireg}(\p_\kk^d(\kk))$ as the group of elements $f\in\Cr_d(\kk)$ such that $f$ and $f^{-1}$ have no base-points in $\p_\kk^d(\kk)$. 
If $\kk$ is of characteristic zero, the analogous statement of Proposition~\ref{main:birdiff} for $\mathrm{Bireg}(\p_\kk^d(\kk))$ holds with the same proof as for $\BirDiff(\p_\R^d(\R))$.
\end{remark}

%%%
\bibliographystyle{abbrv}
\bibliography{bibliography_cu}
%%%

\end{document}